\documentclass[reqno]{amsart}

\usepackage{amsfonts}
\usepackage[all]{xy}
\usepackage{amssymb}
\usepackage{amsmath}
\usepackage{epsfig}
\usepackage{amscd}
\usepackage{graphicx, color}
\usepackage{epstopdf}
\usepackage{amscd}
\usepackage{amssymb}
\usepackage{amstext}
\usepackage[all]{xy}
\usepackage{mathrsfs}

\usepackage{hyperref}

\def\E{\ifmmode{\mathbb E}\else{$\mathbb E$}\fi} 
\def\N{\ifmmode{\mathbb N}\else{$\mathbb N$}\fi} 
\def\R{\ifmmode{\mathbb R}\else{$\mathbb R$}\fi} 
\def\Q{\ifmmode{\mathbb Q}\else{$\mathbb Q$}\fi} 
\def\C{\ifmmode{\mathbb C}\else{$\mathbb C$}\fi} 
\def\H{\ifmmode{\mathbb H}\else{$\mathbb H$}\fi} 
\def\Z{\ifmmode{\mathbb Z}\else{$\mathbb Z$}\fi} 
\def\P{\ifmmode{\mathbb P}\else{$\mathbb P$}\fi} 
\def\T{\ifmmode{\mathbb T}\else{$\mathbb T$}\fi} 
\def\SS{\ifmmode{\mathbb S}\else{$\mathbb S$}\fi} 
\def\DD{\ifmmode{\mathbb D}\else{$\mathbb D$}\fi} 

\newcommand{\del}{\partial}

\newcommand{\ben}{\begin{enumerate}}
\newcommand{\een}{\end{enumerate}}
\newcommand{\be}{\begin{equation}}
\newcommand{\ee}{\end{equation}}
\newcommand{\bea}{\begin{eqnarray}}
\newcommand{\eea}{\end{eqnarray}}
\newcommand{\beastar}{\begin{eqnarray*}}
\newcommand{\eeastar}{\end{eqnarray*}}
\newcommand{\bc}{\begin{center}}
\newcommand{\ec}{\end{center}}

\newtheorem{thm}{Theorem}[section]
\newtheorem{cor}[thm]{Corollary}
\newtheorem{lem}[thm]{Lemma}
\newtheorem{prop}[thm]{Proposition}

\theoremstyle{definition}
\newtheorem{defn}[thm]{Definition}
\newtheorem{rem}[thm]{Remark}

\newtheorem*{thm*}{Theorem}

\numberwithin{equation}{section}

\def\R{{\mathbb R}}

\def\E{{\mathbb E}}
\def\Z{{\mathbb Z}}
\def\C{{\mathbb C}}
\def\R{{\mathbb R}}
\def\P{{\mathbb P}}

\def\N{{\mathbb N}}

\def\11{{\mathbb I}}

\def\H{\mathbb{H}}

\def\delbar{{\overline \partial}}

\def\C{\mathbb{C}}
\def\Z{\mathbb{Z}}

\def\T{\mathbb{T}}

\def\Q{\mathbb{Q}}

\def\E{\ifmmode{\mathbb E}\else{$\mathbb E$}\fi} 
\def\N{\ifmmode{\mathbb N}\else{$\mathbb N$}\fi} 
\def\R{\ifmmode{\mathbb R}\else{$\mathbb R$}\fi} 
\def\Q{\ifmmode{\mathbb Q}\else{$\mathbb Q$}\fi} 
\def\C{\ifmmode{\mathbb C}\else{$\mathbb C$}\fi} 
\def\Z{\ifmmode{\mathbb Z}\else{$\mathbb Z$}\fi} 
\def\P{\ifmmode{\mathbb P}\else{$\mathbb P$}\fi} 
\def\CS{\ifmmode{\mathbb S}\else{$\mathbb S$}\fi} 
\def\DD{\ifmmode{\mathbb D}\else{$\mathbb D$}\fi} 

\def\R{{\mathbb R}}

\def\E{{\mathbb E}}
\def\Z{{\mathbb Z}}
\def\C{{\mathbb C}}
\def\R{{\mathbb R}}

\def\N{{\mathbb N}}
\def\MM{{\mathcal M}}

\def\delbar{{\overline \partial}}

\def\CJ{{\mathcal J}}

\def\CM{{\mathcal M}}

\def\CS{{\mathcal S}}

\def\CU{{\mathcal U}}
\def\CV{{\mathcal V}}

\def\darr#1{\raise1.5ex\hbox{$\leftrightarrow$}
\mkern-16.5mu #1}

\def\roughly#1{\raise.3ex\hbox{$#1$\kern-.75em
\lower1ex\hbox{$\sim$}}}

\def\opname#1{\mathop{\kern0pt{\rm #1}}\nolimits}

\def\Im{\opname{Im}}

\def\dim{\opname{dim}}

\def\Area{\operatorname{Area}}
\def\Image{\operatorname{Image}}
\def\Int{\operatorname{Int}}

\begin{document}
\quad
\vskip1.375truein

\title[Displacement energy of Lagrangian submanifold]
{Displacement energy of compact Lagrangian submanifold from open subset}
\thanks{This work is supported by IBS project \#IBS-R003-D1.}

\author{Yong-Geun Oh}

\address{Center for Geometry and Physics, Institute for Basic Sciences (IBS), Pohang, Korea \&
Department of Mathematics, POSTECH, Pohang, KOREA}
\email{yongoh1@postech.ac.kr}

\begin{abstract} We prove that for any compact Lagrangian submanifold $L$
the Hofer displacement energy for disjointing $L$ from
an open subset $U$ in tame symplectic manifold $(M,\omega)$
is positive, provided $L \cap U  \neq \emptyset$. We also give an explicit
lower bound in terms of an $\epsilon$-regularity type invariant for pseudo-holomorphic curves
relative to $L$ and $U$.
\end{abstract}

\keywords{Lagrangian submanifolds, Hofer displacement energy, relative $\epsilon$-regularity type
invariant, adapted symplectic embedding}

\date{May 10, 2018}

\maketitle

\hskip0.3in MSC2010: 53D05, 53D35, 53D40; 28D10.
\medskip

\tableofcontents

\section{Introduction}

\label{sec:intro}

For a compactly supported Hamiltonian diffeomorphism $\phi$,
Hofer's norm \cite{hofer} is defined to be
\be\label{eq:||phi||}
\|\phi\| = \inf_{H \mapsto \phi} \|H\|
\ee
where $H \mapsto \phi$ means that  $H: M \times [0,1] \to \R$ is a Hamiltonian
such that $\phi = \phi_H^1$, and
\be\label{eq:||H||}
\|H\| = \int^1_0 (\max H_t - \min H_t)\, dt.
\ee
Here $\phi_H^1$ denotes the time-one map of the flow of the Hamilton's equation
$
\dot{z} = X_H(z).
$

\begin{defn}[Displacement energy]\label{defn:disjunct-energy} Let $A, \, B \subset M$ be two
closed subsets. The displacement energy $e(A,B)$ is defined to be
\be\label{eq:eAB}
e(A,B) = \inf_H \{\|H\| \mid A \cap \phi_H^1(B) = \emptyset\}.
\ee
\end{defn}

Clearly if $A \cap B = \emptyset$, then $e(A,B) = 0$. But even when $A \cap B \neq \emptyset$,
$e(A,B) = 0$ may still happen. For example, if $A, \, B$ are submanifolds of $\dim A + \dim B < \dim M$,
then their displacement energy is zero. When $A$ is a compact submanifold of $M$
and $2\dim A= \dim M$, Laudenbach-Sikorav \cite{laud-sikorav} proved that if $A$ is
non-Lagrangian, $e(A,A) = 0$, provided its normal bundle has a section without zeros, i.e.,
as long as there is no topological obstruction to the disjoining.

When both  $A$ and $B$ are Lagrangian submanifolds and a certain form of Floer
homology $HF(A,B)$ is defined and is non-zero, then $e(A,B) =\infty$.
We also refer to the works by Biran \cite{biran,biran2} and Biran-Cornea \cite{biran-cornea}
for a different kind of such an obstruction to displacement of a symplectic ball $B(\lambda)$ in $M$
from a Lagrangian submanifold $L$.

However when there is no such obstruction to the
disjunction for the pair $(A,B)$ and $e(A,B) < \infty$, this question of measuring the
displacement energy, in particular, proving the positivity $e(A,B) > 0$, is a hard problem in general.
This is the question we are pursuing in the present paper:
More specifically, we study such a measurement when $A = \overline U$ with $U$ an open subset,
and $B=L$ is a compact Lagrangian submanifold intersecting $U$, \emph{when $\dim M \geq 4$}.
(The same holds for $\dim M =2$ which is however easy to see unlike the higher dimensional case.)

Incidentally we would like to recall readers that these two cases, those of Lagrangian submanifolds and
of open subsets, are the only general classes of subsets in a symplectic manifold that
are known to admit such lower bounds for displacing the subset from itself.
We refer readers to \cite{hofer,lal-mcduff,oh:dmj,usher} for such a measurement of the open subsets
in terms of various kind of capacities: Hofer-Zehnder capacity \cite{hofer,usher}, Gromov capacity
\cite{lal-mcduff} and spectral capacity \cite{oh:dmj,usher} etc. For the Lagrangian submanifolds,
we refer to \cite{polterov,chekanov,oh:mrl} where such a measurement is made
with respect to the $\epsilon$-regularity type invariant in the spirit of the present paper.

We recall the definition of tame symplectic manifolds:
A symplectic manifold $(M,\omega)$ is called {\it tame}
if it allows an almost complex structure $J_0$ such that the bilinear form
$\omega(\cdot, J_0 \cdot)$ defines a Riemannian metric on $M$ with
bounded curvature and with injectivity radius bounded away from zero.
In this case, we also call {\it tame} the triple $(M,\omega, J_0)$ or
the almost complex structure $J_0$. As usual, when we do our estimates which are implicit mostly
in this paper, we will use various norms always in terms of a fixed such metric.

The following theorem provides an answer to the case of mixture of the two.

\medskip

\noindent \textbf{Main Theorem} {\em Let $(M,\omega)$ be a tame
symplectic manifold and $L \subset M$ be a compact Lagrangian submanifold. Suppose that $U$
is an open subset such that $L \cap U \neq \emptyset$. Then $e(\overline U,L) > 0$.}

\medskip

In fact, we can estimate the displacement energy $e(\overline U,L)$ in terms of an
$\epsilon$-regularity type invariant whose description is briefly in order. We refer
readers to Theorem \ref{thm:enhanced} for the precise statement of the lower bound
for $e(\overline U,L)$.

We start with recalling the (absolute) $\epsilon$-regularity type invariants used
in \cite{oh:imrn}, \cite{chekanov}, \cite{oh:mrl}.
For each tame $J_0$, we define
\beastar
A(J_0;M, \omega) & = &\inf\{ \omega([v])~|~ v: S^2 \to M, \, \text{\rm
non-constant and }\,
\overline{\partial}_{J_0} v = 0 \} \\
A(J_0,L;M,\omega) & = &\inf\{ \omega([w])~|~ w:(D^2,\partial D^2) \to (M,L),\\
&{}& \hskip1in \text{\rm non-constant and }\,
\overline{\partial}_{J_0} w = 0 \}.
\eeastar
It is not difficult to show
$$
A(J_0;M, \omega), \, A(J_0,L;M,\omega) > 0
$$
from the $\epsilon$-regularity theorem and
tameness of $(M,\omega,J_0)$.
(See \cite[Corollary 3.5]{oh:removal} for its proof.)
We then define
\be\label{eq:ALMomega}
A(L;M,\omega) = \sup_{J_0}\min \{A(J_0;M, \omega), A(J_0,L;M,\omega)\}.
\ee
$A(L;M,\omega)$ could be infinity for general compact $L$ and
its finiteness is equivalent to existence of
certain pseudo-holomorphic sphere or a disc attached to $L$. However
it was proved in \cite{chekanov}, \cite{oh:mrl} that $e(L,L) \geq A(L;M,\omega)$.
In particular, if $L$ is displaceable, equivalently, if $e(L,L) < \infty$, then
$A(L;M,\omega) < \infty$ also holds.

For the purpose of the present paper,
we will also need to construct another more refined invariant which is an analog of the
invariant $A(L;M,\omega)$ above by restricting the choice of $J_0$ to those that have
some local reflectional symmetry on a symplectic ball.
Unlike the case of $A(L;M,\omega)$, defining the $\epsilon$-regularity type invariant
relevant to the purpose of the present paper requires some preparation. This kind of invariant in general
was introduced and systematically used by Biran-Cornea in \cite{biran-cornea} in their study of
mixed symplectic packing number. (See \cite[Definition 1.1.1]{biran-cornea} and also
\cite{barraud-cornea}).) We only consider a relative counterpart of the definition of
$A(L;M,\omega)$ for the Lagrangian boundary condition, which
is defined by using the relative version of
isoperimetric inequality for the holomorphic curves with real boundary
condition. We will denote the resulting invariant by $\epsilon(U,L;M,\omega)$.
We postpone the details of its construction till the next section.

The main geometro-analytic framework of our proof of Main Theorem is an adaptation of the one
from \cite{oh:mrl} which gave a simple proof of Chekanov's positivity
theorem \cite{chekanov} of the displacement energy of general
compact Lagrangian submanifold in tame symplectic manifold.
In this article, we use the same cut-off version
of this Floer's perturbed Cauchy-Riemann equation as that of \cite{oh:mrl} and
adapt the scheme used therein
to the current context of our interest in the following way:
\begin{enumerate}
\item We identify non-emptiness of $U \cap L$ and emptiness of intersections $U \cap \phi_H^1(L)$
as an existence criterion of certain solution of Hamiltonian perturbed
Cauchy-Riemann equations.
\item
We then combine some basic energy estimate from \cite{oh:mrl}
with the ideas from
\cite{oh:dmj}, \cite{biran-cornea}
to relate the displacement energy to the \emph{$\epsilon$-regularity-type}
invariant relative to $L$.
\end{enumerate}
In addition to this, the scheme of relating the displacement energy with
the $\epsilon$-regularity type invariants resembles that of the proof of
nondegeneracy of spectral norm of $Ham(M,\omega)$ given in \cite{oh:dmj}.

We would like to thank M. Kawasaki for informing us of some positivity result
he obtained via a study of Lagrangian spectral invariants
for monotone Lagrangian submanifolds whose Floer cohomology is non-trivial.
This was the starting point of our investigation of the question of
displacing general Lagrangian submanifold from an open subset.
We also thank him for interesting discussions on the problem
in the early stage of current research while he was a member of IBS-CGP. Discussion with him
much helped the author crystalizing the main scheme of the proof.

After the paper was posted in arXiv, Jun Zhang attracted our attention to
Usher's previous work \cite[Corollary 4.10]{usher:submanifolds} in which the same positivity
statement is proved with a slightly different kind of lower bound again
by using the framework of \cite{oh:mrl} as in the present paper. We thank
Zhang for alerting us for \cite{usher:submanifolds} and are sorry for
our omission of that article from our attention.

\section{Relative $\epsilon$-regularity type invariants}

\label{sec:e-regularity}

In this section, we explain a direct analog of the
invariant $A(L;M,\omega)$ mentioned in the introduction.
Various kinds of $\epsilon$-regularity type invariants were
introduced in \cite[Section 4]{oh:dmj} and related to the displacement of
symplectic balls. We closely follow
similar scheme therefrom which we adapt to the present relative
context of the pair $(U,L)$.

We start with some general discussion on Darboux-Weinstein chart.
Let $L \subset M$ be a compact
Lagrangian submanifold. Consider the Darboux-Weinstein chart
$
\Phi: \CU \to \CV
$
where $\CU$ is a neighborhood of $L$ in $M$ and $\CV$ is a
neighborhood of the zero section $o_L \subset T^*L$. Then by definition,
we have
$$
\omega = \Phi^*\omega_0, \quad \omega_0 = - d\theta
$$
for the Liouville one-form $\theta$ on $T^*L$ and $\Phi|_L = id_L$
under the identification of $L$ with $o_L$.

Fix any Riemannian metric $g$ on $L$. For $x \in \CU$, we define
$$
\|x\|_{g,\Phi} = \|\Phi(x)\|_{g(\pi(\Phi(x))}
$$
where $\Phi(x) \in T_{\pi(\Phi(x))}^* L$ and $\pi: T^*L \to L$
is the canonical projection, and $\|\cdot \|_{g(q)}$ is the
norm on $T_q^*L$ induced by the inner product $g(q)$.

\begin{defn}\label{defn:width} Let $L \subset M$ be a compact
Lagrangian submanifold equipped with a metric $g$. Consider the Darboux-Weinstein chart
$
\Phi: \CU \to \CV.
$
Define
$$
{\frak w}_{\text{DW}}(\Phi;g): = \inf_{q \in L} \left(\sup_{x \in \pi^{-1}(q) \cap \CU} \|x\|_{g,\Phi}\right)
$$
and
\be\label{eq:DW-width}
{\frak w}_{\text{DW}}(L;M) = \sup_{\Phi} {\frak w}_{\text{DW}}(\Phi;g)
\ee
over all Darboux-Weinstein chart of $L$.
We call ${\frak w}_{\text{DW}}(L;M)$ the \emph{Weinstein width} of $L$
(relative to the metric $g$). We will fix this metric $g$ on $L$ throughout the paper.
\end{defn}
Obviously ${\frak w}_{\text{DW}}(L;M) > 0$
since ${\frak w}_{\text{DW}}(\Phi;g) > 0$ for any Darboux-Weinstein chart $\Phi$
 for compact Lagrangian submanifold $L$.

Next following Biran-Cornea \cite{biran-cornea}, we introduce

\begin{defn}\label{defn:relativetoL} Let $e:(B^{2n}(r),\omega_0) \to (M,\omega)$
be a symplectic embedding of the closed
standard ball of $B^{2n}(r) \subset \C^n$ of radius $r$. We say $e$ \emph{is adapted to $L$}
or simply \emph{$L$-adapted} if
\be\label{eq:e-1L}
e^{-1}(L) = B^{2n}(r) \cap \R^n.
\ee
\end{defn}

We prove the following existence result on such an $L$-adapted embedding.

\begin{prop}\label{prop:adaptedB>0} Let $L$ be a compact Lagrangian submanifold of
$(M,\omega)$ and let $p \in L$. Then there exists an $L$-adapted embedding $e: B^{2n}(r) \to M$
centered at $p$ for some $r > 0$ whose size depends only on the pair $(M,L)$.
\end{prop}
\begin{proof} We first choose a Darboux-Weinstein neighborhood
$\CU$ of $L$ in $M$. Then we take a canonical coordinate
$(q_1, \ldots, q_n,p_1, \ldots, p_n)$ in a neighborhood $V$ of $p$ in $\CU$.
Using this coordinates, we identify $V$ as an open subset of $\C^n$ by
identifying the standard complex coordinates $(z_1,\ldots, z_n)$ to be $z_j = q_j + \sqrt{-1} p_j$.
Then we can obviously find a symplectic embedding $e: B^{2n}(r) \to V$ centered at
$p$ that also satisfies
$$
e(B^{2n}(r)) \cap L = e(B^{2n}(r) \cap \R^n), \quad \R^n \subset \C^n
$$
if we choose a sufficiently small $r > 0$. In other words, the resulting $e$ is $(U,L)$-adapted.

We remark that the choice of such a radius $r > 0$
depends only on the width ${\frak w}_{\text{DW}}(L;M)$ which in turn
depends only on the pair $(M,L)$. This finishes the proof.
\end{proof}

With this definition, we recall the following Biran-Cornea's relative version of Gromov area of
the pair $(M,L)$ from \cite{biran-cornea}.

We denote by $j$ the standard complex structure of $\C^n$
(or on any 2-dimensional Riemann surface in general).

\begin{defn}\label{defn:J-rele} Let $L \subset (M,\omega)$ be a Lagrangian
submanifold and a symplectic embedding $e: B^{2n}(r) \to M$ relative to $L$
be given. We say a compatible almost complex structure $J_0$ adapted to
$e$ if $J_0 = e_*j$ on $e(B^{2n}(r)) \subset M$. We denote by $\CJ_{\omega;e}$
the set of $J_0$ adapted to $e$.
\end{defn}

\begin{defn} Let $L \subset M$ be given. Consider pairs $(e,J_0)$ with $J_0$ adapted to
a symplectic embedding $e$ adapted to $L$. Call it an \emph{adapted pair} of $L$
or simply an $L$-adapted pair.
\end{defn}

Consider a compact surface $\Sigma$ with the decomposition
$$
\del \Sigma = \del_-\Sigma \cup \del_+ \Sigma
$$
so that $\del_- \Sigma \cap \del_+ \Sigma$ consists of a finite number of points.
Call a subset $C \subset M$ a $J_0$-holomorphic curve if we can represent $C$
as the image of a somewhere injective $J_0$-holomorphic map $w: \Sigma \to M$.
We denote $\del C = w(\del \Sigma)$ and $\del_\pm C = w(\del \Sigma_\pm)$.
When we are given an open subset $U$ intersecting $L$, we
consider $e$ that also satisfies $e(B^{2n}(r)) \subset U$.

\begin{defn}\label{defn:L-proper} Let $(e,J_0)$ be an $L$-adapted pair.
\begin{enumerate}
\item We say a $J_0$-holomorphic curve $C \subset M$ is
\emph{properly $(L,e)$-adapted } if $\del_-C \subset L$ and
$\del_+ C \cap e(\del B^{2n}(r)) = \emptyset$.
\item Let $U$ be a given open subset and let
symplectic embedding $e$ satisfy $e(B^{2n}(r))\subset U$. In this case we say
$(e,J_0)$ is \emph{$(U,L)$-adapted}, and the curve $C$ given as above is
$(U,L;e,J_0)$-adapted.
\end{enumerate}
\end{defn}

The following is the relative analog to the monotonicity formula,
which is a consequence of the usual monotonicity formula combined with
a doubling argument via the reflection principle for
$J_0$-holomorphic discs but this time considering only those $J_0$
coming from $\CJ_{\omega;e}$.

\begin{lem}\label{lem:1/2monotonicity} Let $(e,J_0)$ be an $(U,L)$-adapted pair. Then for any properly
$(U,L;e,J_0)$-adapted curve $C$, there exists $r_0> 0$ depending only on $(U,L)$ such that
\be\label{eq:AeJ}
\int_{C\cap e(B^{2n}(r))}  \omega \geq \frac{\pi r^2}{2}
\ee
for all $0 < r \leq r_0$.
\end{lem}
\begin{proof} Represent $C$ as the image of $J_0$-holomorphic map $w:\Sigma \to M$.
Since $(e,J_0)$ is $(U,L)$-adapted, $e^{-1}\circ w$ which is a holomorphic map
into $\C^n$ with respect to standard complex structure $j$ with real boundary condition.

Therefore we can apply the standard reflection principle in complex
one-variable theory to $e^{-1}\circ w$, and double it to a
surface that is reflection-symmetric.
Applying $e$ back to it, we double $C$
to a proper $j$-holomorphic curve $S = \overline C \# C$.
Then $e^{-1}(S)$ defines a proper holomorphic curve in $B^{2n}(r)$ containing $0 \in B^{2n}(r) \cap \R^n$.
Applying the isoperimetric inequality for holomorphic curves $e^{-1}(S)$ in $B^{2n}(r) \subset \C^n$
and the symplectic property of the embedding $e$, we have
$$
\Area(S) \geq \pi r^2.
$$
Since $S = \overline C \# C$ and $\Area(C) = \Area(\overline C)$,
$$
\Area(S) = 2 \Area(C) \leq 2 \int w^*\omega
$$
Combining these two inequalities, we have finished the proof of \eqref{eq:AeJ}.
\end{proof}

Based on this relative monotonicity formula, we proceed the process of defining the
analog to the invariant $A(L;M,\omega)$ relative to an open subset $U$.

For each given properly $(U,L)$-adapted pair $(e,J_0)$, we define
\bea\label{eq:ALeJ0}
A(U,L;e,J_0)
& = & \inf_{C} \left\{\int_{C\cap e(B^{2n}(r))} \omega \, \Big \vert\,
C \text{ is $(U,L;e,J_0)$-adapted}, e(0) \in C \, \right\}. \nonumber\\
&{}&
\eea
We put $A(U,L;e,J_0)  = \infty$  as usual if there exists no $L$-adapted pair $(e,J_0)$ that admits
an $(U,L;e,J_0)$-adapted curve satisfying \eqref{eq:ALeJ0}. Next we define
\be\label{eq:AM}
A(U,L;M,\omega) = \sup_{(e,J_0)} \{A(U,L;e,J_0) \mid \text{$(e,J_0)$ is properly $(U,L)$-adapted}\}.
\ee

Next we introduce a restricted version of $A(J_0;M,\omega)$ and $A(J_0,L;M,\omega)$ given in \eqref{eq:ALMomega}.
 We define $A^U(J_0,L;M,\omega)$ (resp. $A^U(J_0;M,\omega)$) in the same way
as that of $A(J_0,L;M,\omega)$ (resp. of $A(J_0;M,\omega)$) given in the introduction, but
restricting $J_0$ to those contained in $\CJ_{e,\omega}$ for some $(U,L)$-adapted embedding
$e$. Then define
\be\label{eq:AULMomega}
A^U(L;M,\omega) = \sup_{(e,J_0)} \min\{A^U(J_0;M,\omega), A^U(J_0,L;M,\omega)\}
\ee
where we take the supremum over all $(U,L)$-adapted pair $(e,J_0)$.

Finally we are arrived at the definition of the invariant we have been seeking for.
\begin{defn}\label{defn:epsilonLM}

We denote
$$
\epsilon(U,L;M,\omega) = \min \{A^U(L;M,\omega), A(U,L;M,\omega)\}.
$$
\end{defn}
A priori the possibility of $\epsilon(U,L;M,\omega) = \infty$ is not ruled out.
The following theorem will guarantee that this will not happen under
the circumstance of Main Theorem.

\begin{thm}\label{thm:exist} Let $(M,\omega)$ is a tame symplectic manifold
and $L \subset M$ be a compact Lagrangian submanifold. Let $U$ be an open subset
such that $L \cap U \neq \emptyset$ and $\overline U \cap \phi(L) = \emptyset$
for some Hamiltonian diffeomorphism $\phi$. Then
$0 < \epsilon(U,L;M,\omega) < \infty$.
\end{thm}
We will give its proof in the course of proving Main Theorem. The main task is
to establish an \emph{existence result} of a $(U,L)$-adapted pair $(e,J_0)$ that
admits a properly $(L,e)$-adapted $J_0$-holomorphic curve $C$.

One way of producing such a curve $C$ appearing above is as follows.
Consider a map $v: \R \times [0,1] \to M$ that satisfies
$$
v(0,0) \in U\cap L
$$
and the genuine Cauchy-Riemann equation
\be\label{eq:v-CR}
\begin{cases}
\frac{\partial v}{\partial \tau} + J_0\frac{\partial v}{\partial t} =0\\
v(\tau,0) \in L, \, v(\tau,1) \in \phi_H^1(L).
\end{cases}
\ee
We recall that any stationary i.e., any $\tau$-independent finite energy solution of \eqref{eq:v-CR}
is a constant solution valued in $L \cap \phi_H^1(L)$. The following lemma is a key
lemma that enters in our construction of an $(U,L;e,J_0)$-adapted curve which plays an
important role in the proof of Main Theorem.

\begin{lem}\label{lem:v-nonconstant}
Suppose $L \cap U \neq \emptyset$ and $U \cap \phi_H^1(L) = \emptyset$.
Let $(e,J_0)$ be a $(U,L)$-adapted pair.
Then for any finite energy solution $v$ of \eqref{eq:v-CR},
there exists some $R_0 > 0$ such that
\be\label{eq:vatinfty}
\Image v|_{(\R \setminus [-R_0,R_0]) \times [0,1]}\subset M \setminus U.
\ee
\end{lem}
\begin{proof} We prove it by contradiction. Suppose to the contrary that there exists
a sequence $R_j \to \infty$ such that
\be\label{eq:vRj}
\Image v|_{(\R \setminus [-R_j,R_j]) \times [0,1]}\cap \overline U \neq \emptyset.
\ee
Pick a point $q_j$ from $\Image v|_{(\R \setminus [-R_j,R_j]) \times [0,1]}\cap \overline U$
for each $j$.
By choosing a subsequence, if necessary, we can express
$q_j = v(R_j', t_j)$ or $q_j = v(-R_j', t_j)$ for $R_j' \geq R_j$ for $j = 1, \, 2,\, \ldots$.
Without loss of any generality, we may assume $q_j = v(R_j', t_j)$ since the other case
can be treated the same. Again by choosing a subsequence, we may assume
$q_j \to q$ for some point $q \in \overline U$. (Recall we assume that $M$ is tame and
$L$ is compact. It is easy to derive from the montonicity formula that the image
$\Image v$ is bounded and so the set \eqref{eq:vRj} is pre-compact.)

Then we consider the path $z_j:[0,1] \to M$ defined by
$$
z_j(t) = v(R_j',t).
$$
On the other hand, by the finite energy condition of $v$, we have
$$
\lim_{j\to \infty} E_{J_0}\left(v|_{(\R \setminus [-R_j,R_j]) \times [0,1])}\right) = 0.
$$
Using the standard $\epsilon$-regularity theorem
(see \cite[Proposition 3.3]{oh:removal} for example) applied to $v$
on the domains of the uniform size
$$
[R_j'-1, R_j +1] \times [0,1] \cong [-1,1] \times [0,1]
$$
we obtain a convergence $\|\dot z_j\|_{C^0} \to 0$ of the $C^1$-norm of $z_j$
as $j \to \infty$. Therefore since $q_j = z_j(t_j)$ and $q_j \to q$, this implies
$z_j$ uniformly converges to a constant path $z$ valued at $q$, i.e.,
$z(t) = q$ for all $t \in [0,1]$.
Furthermore the boundary condition
$$
v(\tau,0) \subset L, \quad v(\tau,1) \subset \phi_H^1(L)
$$
of $v$ also implies $z_j(0) \in L$ and $z_j(1) \in \phi_H^1(L)$.
This implies $q \in L \cap \phi_H^1(L)$.

Combining the above, we conclude that $q \in L \cap \phi_H(L) \cap \overline U$
which contradicts to the hypothesis $\overline U \cap \phi_H^1(L) = \emptyset$.
This finishes the proof.
\end{proof}

\begin{rem} If we know that $v$  uniformly converges as $\tau \to \pm \infty$
as in the case of transversal intersection $L \pitchfork \phi_H^1(L)$,
we can simply write as $v(\pm \infty) \in M \setminus U$ instead of \eqref{eq:vatinfty}
in the statement of Lemma \ref{lem:v-nonconstant}. Since we do not
impose this transversal intersection property, the statement of this lemma
is the only thing we can achieve for the general case. This will be enough for our purpose.
\end{rem}

Then the curve $C = \Image v$ is one that can be used in
Theorem \ref{thm:exist}, which will then prove finiteness of $\epsilon(U,L;M,\omega)$.

Furthermore we also have the following lower bound of the symplectic area of
the curve $v$.

\begin{prop}\label{prop:case-exist} Under the same hypotheses as in Lemma \ref{lem:v-nonconstant}
any finite energy solution $v$ of \eqref{eq:v-CR} with $v(0,0) = p$ satisfies
\be\label{eq:v-area}
\int v^*\omega \geq \frac{\pi r^2}{2}.
\ee
\end{prop}
\begin{proof}
By the hypothesis $\phi_H^1(L) \cap \overline U = \emptyset$,
$$
v(\tau,1) \in M \setminus U
$$
for all $\tau \in \R$ since we have $v(\tau,1) \in \phi_H^1(L)$ by the boundary condition at $t = 1$.
By Lemma \ref{lem:v-nonconstant}, $v$ can not be a constant map since $v(0,0) = p \in L$.

Furthermore by $(U,L)$-adaptedness of the embedding $e$ and since $v(0,0) = p$,
we have
$$
e(B^{2n}(r) \cap \R^n) = e(B^{2n}(r))\cap L \subset U \cap L.
$$
Then $v$, restricted to the connected component of
$$
v^{-1}\left(\Int e(B^{2n}(r)) \cap \Im v\right) \subset \R \times [0,1]
$$
containing the point $(0,0)$, defines a surface
$C$ that is $(U,L;e,J_0)$-adapted.
Now Lemma \ref{lem:1/2monotonicity} finishes the proof.
\end{proof}

Therefore we would like to produce a $J_0$-holomorphic map $v$ used in Proposition \ref{prop:case-exist}.
For this purpose, we exploit the correspondence between the dynamical version
and the geometric version for the Lagrangian intersection Floer equations in the
spirit of \cite{oh:dmj} where this correspondence was extensively used for
the applications of spectral invariants to the geometry of Hamiltonian
diffeomorphism group $Ham(M,\omega)$.

Let $H = H(t,x)$ be any given compactly supported Hamiltonian and denote $\phi = \phi_H^1$.
We require $J$ to satisfy the condition
\be\label{eq:JJ0}
J(t,x) = (\phi^t_H)^*J_0
\ee
and consider the associated perturbed Cauchy-Riemann equation
\be\label{eq:CRJt}
\begin{cases}
\frac{\partial u}{\partial \tau} + J_t(\frac{\partial u}{\partial t} - X_H(u)) =0\\
u(\tau,0) \in L, \, u(\tau,1) \in L.
\end{cases}
\ee
 Let $u$ be any such solution of \eqref{eq:CRJt} and $v$ be the map defined by
\be\label{eq:uv}
v(\tau,t) = \phi_H^t(u(\tau,t)).
\ee.
The associated energy is given by
$$
E_{(J,H)}(u) = \frac{1}{2} \int\int \left|\frac{\del u}{\del \tau}\right|_J^2 +
\left|\frac{\del u}{\del t} - X_H(u) \right|_J^2\, dt\, d\tau.
$$
The following lemma is standard, which follows from direct calculation.

\begin{lem}\label{lem:energyequality} Let $J_t = (\phi_H^t)^*J_0$.
For a given finite energy solution $u$ of \eqref{eq:CRJt}, consider the map $v:[0,1] \times \R \to M$
the map as defined in \eqref{eq:uv}. Then $v$ satisfies \eqref{eq:v-CR} and
\be\label{eq:energy=area}
E_{(J,H)}(u) = E_{J_0}(v) = \int v^*\omega.
\ee
\end{lem}

An immediate corollary of the above discussion is the following

\begin{cor}\label{cor:epsilon>0} Let $\overline U \cap \phi_H^1(L) = \emptyset$
and $p \in U \cap L$. Suppose there exists a solution $u$ of \eqref{eq:CRJt} satisfying $u(0,0) = p$.
Then
\be\label{eq:epsilon>}
\epsilon(U,L;M,\omega) > 0.
\ee
\end{cor}
\begin{proof} By Proposition \ref{prop:adaptedB>0},  there exists a $(U,L)$-adapted embedding
$e: B^{2n}(r) \to M$, i.e., one satisfying
$$
e(B^{2n}(r)) \subset U, \quad e(B^{2n} \cap \R^n) \subset U \cap L
$$
for some $r > 0$. Then we consider the map $v$ defined as in \eqref{eq:uv}. This map
$v$ satisfies the conditions given in Proposition \ref{prop:case-exist}.
In particular existence of such a map $v$ proves $\epsilon(U,L;M,\omega) > 0$.
\end{proof}

With this definition of $\epsilon(U,L;M,\omega)$, here is the precise version of
Main Theorem.

\begin{thm}\label{thm:enhanced} Let $(M,\omega)$ be a tame symplectic manifold and $L \subset M$ be a
 compact Lagrangian submanifold. Suppose $U \cap L \neq \emptyset$. Then
$$
e(\overline U,L) \geq \epsilon(U,L;M,\omega).
$$
\end{thm}

In the rest of paper, we give the proof of Theorem \ref{thm:enhanced}.
Along the way, we will also prove Theorem \ref{thm:exist}.

\begin{rem}
In practice, the usage of this theorem is two-fold as in \cite{oh:dmj}:
one is for the lower bound
for the displacement energy between $L$ and $U$ and the other is for the
upper bound for the areas of relevant pseudoholomorphic curves. The latter
measures the maximal possible \emph{size} of the open subset $U$
displaceable from $L$ through the chain of inequalities
$$
e(\overline U,L) \geq \epsilon(U,L;M,\omega) \geq \frac{\pi r^2}{2}
$$
for $(U,L)$-adapted symplectic embedding $e: B^{2n}(r) \to M$
displaceable from $L$.
\end{rem}

\section{Cut-off perturbed Cauchy-Riemann equations}
\label{sec:cut-off}

In this section, we largely borrow verbatim the basic framework
that was used by the author in \cite{oh:mrl} for the study of displacement energy
of compact Lagrangian submanifold from itself.

We first recall the well-known correspondence between the
Lagrangian intersections $\phi_H^1(L) \cap L$ and the set of
Hamiltonian chords of $L$.
Let $\phi$ be a Hamiltonian diffeomorphism of $(M,\omega)$.
Let $L$ be a compact Lagrangian submanifold. We have
one-one correspondence between $L\cap \phi(L)$ with the set of
solutions $z:[0,1] \to M$ of
\be\label{hamode}
\dot{z} = X_H(t,z), \quad z(0),z(1) \in L.
\ee
Here is the precise correspondence:
\be
p \in L\cap \phi(L) \longleftrightarrow z = z^H_p \,\mbox{with } \,
z^H_p(t):= \phi_H^t((\phi_H^1)^{-1}(p)).
\ee

Following \cite{oh:mrl}, for each $K \in \R_+ = [0,\infty)$, we define a function $\rho_K:
\R \to [0,1]$ as follows: for $K \geq 1$, we define
$$
\rho_K(\tau) =
\begin{cases}0 \hskip.2in  \text{\rm for }
       &\,  |\tau| \geq K+1 \\
                         1 \hskip.2in  \text{\rm for }
 &\,  |\tau| \, \leq K
 \end{cases}
$$
with
\beastar
\rho^\prime_K < 0  \quad &\mbox{\rm for }& ~  K < \tau < K+1 \\
\quad     > 0  \quad &\mbox{\rm for }&~  -K-1 < \tau < -K
\eeastar
and  for $ 0\leq K \leq 1$,
$$
\rho_K = K \cdot \rho_1.
$$
In particular, $\rho_0 \equiv 0$.

Let $H: [0,1] \times M \to \R$ be a Hamiltonian such that
\be
U \cap \phi^1_H(L) = \emptyset,
\ee
i.e., such that the equation
$$
\begin{cases} \dot{z} = X_H(z) \\
z(0), \, z(1) \in L
\end{cases}
$$
has no solution satisfying
$$
z(0) \in L, \quad z(1) \in U.
$$
Then we consider
a three-parameter family $J = \{J_{(K,\tau, t)}\}_{(K,\tau,t) \in \R \times \R\times [0,1]}$
of tamed almost complex structures such that
\be\label{eq:J}
J_{(K,\tau,t)} = \begin{cases} J_0 \quad & \text{for $|\tau|$ sufficiently large
or for $t=0,1$}\\
(\phi_H^t)^*J_0 \quad & \text{for } \, -K \leq \tau \leq K
\end{cases}
\ee
where $J_0$ is a fixed (genuine) almost complex structure on $M$ that is
tamed to $\omega$.

We would like to remark that it is necessary to vary
almost complex structures in terms of $t$ to get appropriate
transversality result for the Floer complex (see [FHS], [O3] for
detailed account of the transversality proof).

Throughout this paper, we will exclusively
denote by $J_0$ any (genuine) almost complex structure and by $J$
a (domain dependent) two-parameter version of them.
We denote a one-parameter family of them by
$$
\overline{J} = \{J_K\}_{K \in [0, +\infty)}
$$
such that
\be\label{eq:K>>0} J_K= J_\infty =J_\infty(\tau,t) \quad \text{for sufficiently large $K$.}
\ee
For each such pair $(\overline{J},H)$,
we consider one parameter family of perturbed Cauchy-Riemann equations
for the map $u:\R\times [0,1]\to M$,
\be\label{eq:CRJKH}
\begin{cases}
{\partial u\over\partial\tau}+J_K(\tau,t,u)\Big({\partial u\over\partial
t}-\rho_K(\tau) X_H(u)\Big)=0 \\
u(\tau , 0),\; u(\tau ,1)\in L
\end{cases}
\ee
for each $K \in \R_+$.
\begin{rem}
 This equation should be regarded as
the one used for the chain isotopy in the Floer homology theory
connecting the Hamiltonian $0$ to $H$ and then to $0$ back.
\end{rem}
The relevant energy of general smooth map $u: \R \times [0,1] \to M$
for the equation \eqref{eq:CRJKH} is given by
$$
E_{(J_K,H)}(u) = \frac12 \int\int \Big|{\partial
u\over\partial\tau}\Big|^2_{J_K} + \Big|{\partial
u\over\partial t} - \rho_K(\tau) X_H(u)\Big|^2_{J_K}\, dt\, d\tau.
$$
We will be interested in the solutions of \eqref{eq:CRJKH}
with finite energy. We note that the energy is reduced to
\be\label{eq:EJK}
E_{(J_K,H)}(u):=\int^\infty_{-\infty}\int^1_0 \Big|{\partial
u\over\partial\tau}\Big|^2_{J_K}dtd\tau <\infty
\ee
for a solution $u$ of \eqref{eq:CRJKH}. We denote by
$
\CM_K(\overline J,H)
$
the set of finite energy solutions thereof.

Noting that $\R\times [0,1]$ is conformally isomorphic to $D^2\backslash\{
-1,1\}$, it follows from the choice of the cut-off function $\rho_K$
that \eqref{eq:CRJKH} and \eqref{eq:EJK} imply that the map
$$
u\circ\varphi :(D^2\backslash\{ -1,1\},\;\partial
D^2\backslash\{ -1,1\})\to (M,L),
$$
with a conformal diffeomorphism $\varphi$,
has finite (harmonic) energy and $J_0$-holomorphic near $\{ -1,1\}$.
Then the removable
singularity theorem [O1]  enables us to extend this to the whole disc, which
we denote by
$$
\widetilde u:(D^2, \partial D^2)\to (M,L)
$$
is smooth. We denote by $[u]\in\pi_2 (M,L)$
the homotopy class defined by $\widetilde u$.

Now for each $K \in  \R_+ $ and for $A\in \pi_2(M,L)$,
we study the following moduli space
\bea\label{eq:MMKA}
\MM_K(\overline{J},H;A) & = & \{u : \R \times [0,1]  \to M
\emph{}~|~ u \, \text{\rm satisfies \eqref{eq:CRJKH} }, \, E_{(J_K,H)}(u) < \infty \nonumber \\
& {} & \quad \text{\rm and }~  [u] = A \text{ in } \, \pi_2(M,L) \}.
\eea
Since \eqref{eq:CRJKH} is a compact perturbation of the standard pseudo-holomorphic
equation of discs with Lagrangian boundary condition, the standard
index formula from [G] implies
$$
\text{\rm dim } \MM_K(\overline{J},H;A) = \mu_L(A) + n
$$
for generic $\overline J, H$, provided it is non-empty. Here $n$ is the dimension
of the Lagrangian submanifold $L$ and $\mu(A)$ is the Maslov index of the map
$u: (D^2,\del D^2) \to (M,L)$ in class $[u] = A$. Then we have the decomposition
$$
\CM_K(\overline J,H) = \bigcup_{A \in \pi_2(M,L)} \CM_K(\overline J,H;A).
$$

\medskip

\begin{lem}\label{lem:MM0K} $\MM_K(\overline{J},H;A)$ for $K=0$, $A = 0$ in $\pi_2(M,L)$
consists of constant solutions and is Fredholm regular
for any almost complex structure $J_0$. In particular
$\MM_0(\overline{J},H;0)$ is diffeomorphic to $L$.
\end{lem}
\begin{proof} Let $u \in \MM_0(\overline{J},H;0)$. Recall that
for $K = 0$ \eqref{eq:CRJKH} becomes
$$
{\partial u \over \partial \tau} + J_0 {\partial u\over \partial t} = 0, \quad u(\tau,0), \, u(\tau,1) \subset L.
$$
Since $[u] = 0$, the associated disc $\widetilde u$ above must be constant.
The Fredholm regularity of constant solutions is not difficult to check
and is well-known. The last statement follows by considering the evaluation map
$ev: \CM_0(\overline J,H;0) \to L$ given by $ev(u) = u(0,0)$.
\end{proof}

We next state a simple but a fundamental a priori energy bound for any element
$u: \R \times [0,1] \to M$ of the moduli space \eqref{eq:MMKA}, whose proof is
given in the proof of Lemma 2.2 \cite{oh:mrl}. (See also Remark 2.3 therein.)
We omit its proof here referring readers to \cite{oh:mrl}
or to \cite[Lemma 11.2.6]{oh:book1} for the details of the proof.

\begin{lem}\label{lem:E-bound} For all $K \geq 0$ and $A \in \pi_2(M,L)$, we have
\be\label{eq:EJKbound}
E_{(J_K,H)}(u) \leq \omega(A) + \|H\|
\ee
if $[u] = A \in \pi_2(M,L)$. In particular, when $A = 0$, we have
\be\label{eq:EJK-0}
E_{(J_K,H)}(u) \leq \|H\|.
\ee
\end{lem}
Note that Lemma \ref{lem:MM0K} and \eqref{eq:EJKbound} hold for {\it any} $\overline{J}$ and $H$ that
satisfy \eqref{eq:CRJKH} and \eqref{eq:EJK} respectively.
Therefore we can do the standard Fredholm theory and
the genericity arguments with such pairs $(\overline{J}, H)$.
(See also Remark \ref{rem:transversality} (1) below.) We will always
carry out this standard genericity argument without further mentioning
details, whenever necessary.

\section{Creating a Hamiltonian chord}
\label{sec:creating}

Let $H: [0,1] \times M \to \R$ be a given (generic) Hamiltonian.
For generic $\overline{J} = \{J_K\}_{K \in [0, +\infty)}$
satisfying (2.3), we form the parameterized moduli space
$$
\MM^{\text{\rm para}}(\overline{J},H;A) := \bigcup_{K \in \R_+} \{K\} \times
\MM_K(\overline{J},H;A)
$$
for each $A \in \pi_2(M,A)$. It becomes a smooth manifold of dimension $\mu(A) +n+1$ with boundary
by the parameterized version of the index theorem (2.7).
We consider the evaluation map
\be\label{eq:Ev0}
Ev_A : \MM^{\text{\rm para}}(\overline{J},H;A) \times \R \to L\times \R_+\times \R ;
\quad (u, K,\tau) \mapsto (u(\tau, 0), K, \tau).
\ee
We also consider the evaluation map
\be\label{eq:eviAK}
ev_{A,K}^i: \CM_K(\overline J,H;A) \to L;\quad
ev^i_{K,A}(u) = u(0,i), \quad i=0,\, 1.
\ee
The following is the main ingredient in our proof. This proposition is
the counterpart of Lemma 2.2 \cite{oh:mrl}.

\begin{prop}\label{prop:alternative} Let $K_\alpha \to \infty$ and let
$u_\alpha \in \MM_{K_\alpha}(\overline{J},H)$ be a sequence satisfying
$u_\alpha(0,0) = p$ for a given $p \in L$ and the energy bound
\be\label{eq:E<C}
E_{(J_{K_\alpha},H)}(u_\alpha) < C
\ee
for all $\alpha$ for a constant $C > 0$. Then
\begin{enumerate}
\item either the given point $p$ is contained in $u_0(\R \times \{0\})$ for some
non-stationalry solution $u_0$ of
\eqref{eq:CRJt},
\item or there exists a non-constant
$J_0$-holomorphic disc $w:(D^2,\del D^2) \to (M,L)$
with $w(\del D^2) \subset L$ with $p \in w(\del D^2)$.
\end{enumerate}
\end{prop}
\begin{proof}
Using the a priori bound \eqref{eq:EJK-0}, we study a local limit of $u_\alpha$.
More specifically, we consider a limit of the sequence
$$
u_\alpha|_{[-K_\alpha,K_\alpha] \times [0,1]}
$$
on every compact subset of $\R \times [0,1]$ as $\alpha \to \infty$,
by taking a subsequence if necessary.
By the energy bound \eqref{eq:E<C}
and Dominated Convergence Theorem,
$$
E_{(J_{K_\alpha},H)}\left(u|_{\R \setminus [-K_\alpha,K_\alpha] \times [0,1]}\right) \to 0
$$
as $\alpha \to \infty$. We recall the readers from \eqref{eq:K>>0}
that $J_K = J_\infty$ for all sufficiently large $K$.
Therefore by the choice of $\rho_K$ and $J_K$ that $u_\alpha$ satisfies the equation
$$
\begin{cases}
\frac{\del u_\alpha}{\del \tau} + J_t\left(\frac{\del u_\alpha}{\del t} - X_H(u_\alpha)\right) = 0 \\
u_\alpha(\tau,0), \, u_\alpha(\tau,1) \in L
\end{cases}
$$
on $[-K_\alpha,K_\alpha]\times [0,1]$.

Then via the standard diagonal subsequence argument, the energy bound \eqref{eq:EJK-0}
and Gromov-Floer compactness theorem (or rather the way how it is proved) applied to
$u_\alpha|_{[-K_\alpha,K_\alpha] \times [0,1]}$ produce a limit
$u_\infty$ that has the decomposition
$$
u_\infty = u_0 + \sum_i v_i + \sum_j w_j
$$
for a collection of $J_{\infty,(\tau_i,t_i)}$-holomorphic spheres ${\bf v} = \{v_i\}_{i=1}^{N_1}$
with some $(\tau_i,t_i)\in \R \times [0,1]$, and a collection ${\bf w} = \{w_j\}_{j=1}^{N_2}$ of
$J_0$-holomorphic discs $w_j: (D^2,\del D^2) \to (M,L)$ respectively. And
$u_0: \R \times [0,1] \to M$ is a uniform limit of $u_\alpha$ on compact
subsets of $\R \times [0,1]$ modulo bubbling and satisfies the equation
$$
\begin{cases}
\frac{\partial u}{\partial\tau}+J_\infty(t,u)\left(\frac{\partial u}{\partial t}-X_H(u)\right)=0\\
u(\tau , 0),\; u(\tau ,1)\in L.
\end{cases}
$$
We also have the energy bound
\be\label{eq:EJinfty}
E_{(J_\infty,H)}(u_0) + \sum_j \omega([w_j]) + \sum_i \omega([v_i]) \leq C
\ee
where $\omega([w_j]),\, \omega([v_i])$ are the symplectic areas.
There are two alternatives:
\begin{enumerate}
\item Either the given point $p$ is contained in $u_0(\R \times \{0\})$,
\item or the point is contained in one of the disc bubbles $w_j$.
\end{enumerate}
(We recall that $p$ is contained in boundary of $\R \times [0,1]$.)
In Case (1), $u$ cannot be a stationary solution $u(\tau,t) \equiv z(t)$ where $z$ is a
Hamiltonian chord of $L$. For otherwise, the map $v$ defined by
\eqref{eq:uv} would be also stationary which means $v(\tau,t)\equiv \text{const.}$, and hence
$
v(0,1) = v(0,0).
$
Therefore we would have the chain of identities
$$
p = u(0,0) = v(0,0) = v(0,1) = \phi_H^1(u(0,1))
$$
which implies $p \in \phi_H^1(L) \cap \overline U $. This
would contradict to $\phi_H^1(L) \cap \overline U = \emptyset$.
This proves that $u_0$ cannot be stationary.

In Case (2), it follows from the way how the convergence modulo the bubbles is derived that
there is a $J_0$-holomorphic disc $w:(D^2,\del D^2) \to (M,L)$
satisfying $w(\del D^2) \subset L$ and $p \in w(\del D^2)$.
This finishes the proof.
\end{proof}

An immediate consequence of \eqref{eq:EJinfty} is the following

\begin{cor}\label{cor:alternative} Suppose $\overline U \cap \phi_H^1(L) = \emptyset$ and
$p \in U \cap L$. If there is a sequence
$K_\alpha \to \infty$ such that $\MM_{K_\alpha}(\overline{J},H) \cap ev_{0,K_\alpha}^0(p) \neq \emptyset$, then
$$
\|H\| \geq \epsilon(U,L;M,\omega).
$$
\end{cor}
\begin{proof} We start with \eqref{eq:EJinfty} with $C = \|H\|$ recalling
 the energy estimate \eqref{eq:EJK-0}, and the above two alternatives.

We consider the case (1) first, say $u_0(\tau_0,0) = p$
for some $\tau_0 \in \R$. Consider the map $v_0: \R \times [0,1] \to M$ defined by
$$
v_0(\tau,t) = \phi_H^t(u_0(\tau,t)).
$$
Then $v_0$ satisfies $\delbar_{J_0} v_0 = 0$ by the choice \eqref{eq:JJ0}. Furthermore
$$
v_0(\tau_0,0) = p \in  L \cap U, \quad v_0(\tau,1) =  \phi_H^1(u_0(\tau,1)) \in \phi_H^1(L).
$$
Therefore we have created a $J_0$-holomorphic curve $C$ which is properly
$(U,L)$-adapted. By definition of $A(U,L;M,\omega)$, we have obtained
\be\label{eq:case1}
\|H\| \geq E_{(J_\infty,H)}(u_0) = \int v_0^*\omega \geq \int_C \omega \geq A(U,L;M,\omega).
\ee
Now consider the case (2). Then $p$ is  in the image of $w_j$
for some $j$. Note that if $w_j$ contains $p$, then it must also hold that $w_j(z_0) = p$
for some $z_0 \in \del D^2$ and $w_j(\del D^2) \subset L$.
Therefore  we can take $C$ to be the connected component of
$\Im w_j$ containing $p$. This time we have
\be\label{eq:case2}
\|H\| \geq \int w_j^*\omega \geq A^U(L;M,\omega).
\ee
Combining \eqref{eq:case1} and \eqref{eq:case2}, we have finished the proof
by the definition $\epsilon(U,L;M,\omega)$ in Definition \ref{defn:epsilonLM}.
\end{proof}

Another corollary of the existence result stated in Proposition \ref{prop:alternative} is the
following positivity result whose proof has been postponed in Theorem \ref{thm:exist}, until now.

\begin{cor} Under the hypotheses of Theorem \ref{thm:exist}, we have
$$
0< \epsilon(U,L;M,\omega) < \infty.
$$
\end{cor}
\begin{proof} Take $C = \|H\|$ in Proposition \ref{prop:alternative}. For the case (1),
we have $A(U,L;e,J_0) < \infty$ and for the case of (2)
$A(L;J_0,\omega) < \infty$. This proves $\epsilon(U,L;M,\omega) < \infty$.
On the other hand, $0< \epsilon(U,L;M,\omega)$
follows from Corollary \ref{cor:epsilon>0}.
\end{proof}

\section{Proof of the main theorem}
\label{sec:proof}

We go back to the situation of the Main Theorem, where $L$ is displaceable from $U$
and $L \cap U \neq \emptyset$ so that
there exists a Hamiltonian $H$ such that $\phi_H^1(L) \cap \overline U = \emptyset$.
Let $p \in L\cap U$, fix a symplectic embedding
$e: B^{2n}(r) \to M$ with $p = e(0)$ adapted to $(U,L)$. Then we consider  the set $\CJ_{\omega;e}$ of
almost complex structures adapted to $e$ and form the time-dependent
family $J = \{J_t\}$ and then $\overline J = \{J_K\}$ as in Section \ref{sec:cut-off}.

The following is the basic structure theorem of $\MM_K(J,H;A)$ whose proof
is standard and so is omitted.

\begin{prop}\label{prop:MMKA}
\begin{enumerate}
\item For each fixed $K>0$, there exists a generic choice of
$(\overline J,H)$ such that $\MM_K(\overline J,H;A)$ becomes a smooth manifold of dim
$n + \mu_L(A)$ if non-empty. In particular, if $A=0$, dim
$\MM_K(\overline J,H;A) =n$ if non-empty.
\item For the case $A=0,K=0$, all solutions are constant and Fredholm regular
and hence $\MM_K(\overline J,H;A) \cong L$. Furthermore the evaluation map
$$
ev_{0,0}^0 : \MM_0(\overline J,H;0) \to L :  u \mapsto u(0,0)
$$
is a diffeomorphism.
\item Let $K_0 > 0$ and assume $\MM_{K_0}(\overline J,H;A)$ is regular. Then the parameterized moduli space
$$
\MM_{[0,K_0]}^{para}(\overline J,H;A) := \bigcup_{K \in [0,K_0]} \{K\} \times \MM_K(\overline J,H;A) \to [0,K_0]
$$
is a smooth manifold with boundary, \emph{not necessarily compact}, given by
$$
\left(\{0\} \times \MM_0(\overline J,H;A)\right) \coprod \left(\{K_0\} \times \MM_{K_0}(\overline J,H;A)\right)
$$
and the evaluation map
$$
Ev_A :\MM_{[0,K_0]}^{para}(\overline J,H;A) \times \R \to L \times \R_+ \times
\R : ((K,u), \tau) \mapsto (K,u(\tau,0),\tau)
$$
is smooth.
\end{enumerate}
\end{prop}

\begin{rem}\label{rem:transversality}
\begin{enumerate}
\item We note that the $J_0$ we are using comes from
$\CJ_{\omega;e}$ not from the set of \emph{all} compatible almost complex
structures. Therefore we need to make sure the standard transversality proof
such as \cite{floer, FHS,oh:jga} can be applied
for this restricted class. But this can be seen from the fact that there is no
non-constant solution of \eqref{eq:CRJKH} or \eqref{eq:CRJt} whose image is entirely contained in
$e(B^{2n}(r))$. (See p.323-324 \cite{oh:imrn}, especially the top of p.324 for the explanation
in a similar context.)
\item We also mention that the way how we present our proof, especially Proposition
\ref{prop:alternative}, is deliberately devised so that no transversaility result for the bubbles
either of the spheres or of the discs, nor the intersection theory between
the principal components with bubbles enter in the proof. Only the Gromov-Floer compactness and
the transversality result mentioned in Lemma \ref{lem:MM0K} and the one in (1) of this remark are used.
Both transversality results are easy and standard.
This enables us to dispose any virtual cycle technique and
any kind of positivity hypothesis of Lagrangian submanifolds or of symplectic
manifolds both in the statement of and in the proof of our main theorem.
\end{enumerate}
\end{rem}

With this discussion above in mind, we use a priori bound \eqref{eq:EJK-0}
to apply the following Gromov-Floer compactness theorem \cite{gromov}, \cite{ye},
\cite{floer}, \cite{fukaya-ono}.

\begin{prop}\index{Gromov-Floer compactness theorem}\label{prop:MMAK}
Let $K_\alpha$ with $\alpha =1, \cdots $ converging to $K' \in \R_+ $
and $u_\alpha$ be a sequence of solutions of \eqref{eq:CRJKH} for $K = K_\alpha$ with uniform
bound
$$
E_{J_{(K_\alpha,H)}} (u_\alpha) < C \;\; {\rm for}\;\; C \;\;
      \mbox{independent of} \;\;  \alpha.
$$
Then there exist a subsequence again enumerated by $u_\alpha$ and a
cusp-trajectory $(u_0, {\bf v}, {\bf w})$ such that
\begin{enumerate}
\item  $u_0$ is a solution of \eqref{eq:CRJKH} with $K = K'$.
\item   ${\bf v} = \{ v_i \}^{N_1}_{i=1}$ where each $v_i$ is a
$J_{(\tau_i, t_i)}$-holomorphic sphere and ${\bf w} = \{w_j\}_{j=1}^{N_2}$ each $w_j$ is a
$J_0$-holomorphic disc with boundary lying on $L$.
\item We have the convergence
$$
\lim_{\alpha \to \infty} E_{J_{(K_\alpha,H)}}
(u_\alpha) = E_{J_{K'}} (u_0) + \sum_i \omega ([v_i]) +
\sum_j \omega ([w_j]).
$$
\item   And $u_\alpha$ converges to $(u_0, {\bf v}, {\bf
w})$ in Hausdorff topology and converges in compact $C^\infty$ topology away
from the nodes.
\end{enumerate}
\end{prop}

\begin{proof}[Wrap-up of the proof of Theorem \ref{thm:enhanced}]
For the simplicity of notations, we will just denote $ev_K = ev_{0,K}^0$
defined in \eqref{eq:eviAK}.

We consider two cases separately.
The first case is when the hypothesis of Proposition \ref{prop:alternative} holds so that
there exists a sequence  $K_\alpha \to \infty$ for which
$$
\MM_{K_\alpha}(\overline J, H) \cap E_{(J_{K_\alpha},H)}^{-1}([0,C]) \cap ev_{K_\alpha}^{-1}(p) \neq \emptyset
$$
with the choice of
$$
C = \|H\|.
$$
Then Proposition \ref{prop:case-exist},
\eqref{eq:epsilon>} and Corollary \ref{cor:alternative} already prove Theorem \ref{thm:enhanced}.

Therefore it remains to consider the case where there is a constant $K_0$ for which
\be\label{eq:defnK0}
\MM_K(\overline{J},H) \cap E_{(J_K,H)}^{-1}([0, C]) \cap ev_{K}^{-1}(p) = \emptyset
\ee
for all $K \geq K_0$. We fix one such $K_0 > 0$ in the rest of the proof.
In particular, we have
$$
\MM_K(\overline{J},H;0) \cap ev_{K}^{-1}(p) = \emptyset
$$
by the energy estimate \eqref{eq:EJK-0}.

Now consider an embedded small loop $\gamma:[0,1] \to L$
such that
\be\label{eq:im-gamma}
p= \gamma(0) = \gamma(1), \quad \Image \gamma \subset L \cap U.
\ee
We may choose $\gamma$ so that $\Image \gamma$ is as close to $p$ as we want.

Then we consider a smooth embedded path $\Gamma: [0,1] \to L \times \R_+ \times \R$ with
$$
\Gamma(s) = (\gamma(s), K(s), \tau(s))
$$
such that
\be\label{eq:K0K1}
K(0) =0, \quad \text{\rm and}
\quad  K_0  \leq  K(1) \leq 2K_0.
\ee
Recall that $N_\Gamma$ is regular at $s = 0, \, 1$: This is because
\be\label{eq:defnK0-2}
\MM_{K(1)}(\overline{J},H;0) \cap ev_{K(1)}^{-1}(p) = \emptyset
\ee
by the choice of $K_0$ in \eqref{eq:defnK0} for which the regularity statement is vacuous.
On the other hand for $s=0$, since $K(0) = 0$
$$
\MM_{K(0)}(\overline{J},H;0) \cap ev_{K(0)}^{-1}(p)
= \MM_0(\overline{J},H;0) \cap ev_{0}^{-1}(p).
$$
But the latter intersection
consists of a single element which is a constant map.
This constant map is regular by Lemma \ref{lem:MM0K}.

Applying  Proposition \ref{prop:MMKA} and
the transvesality extension theorem, for a generic choice of $\Gamma$,
we can make
the map \eqref{eq:Ev0} transversal to the path $\Gamma$
so that $N_\Gamma:= Ev_0 ^{-1}(\Gamma)$ becomes a one
dimensional manifold with its boundary consisting of
$$
\MM_{K(0)}(\overline{J},H;0) \cap ev_0^{-1}(p)
\coprod \MM_{K(1)}(\overline{J},H;0) \cap ev_{K(1)}^{-1}(p).
$$
However the second summand is empty by \eqref{eq:defnK0-2} and so $\del N_\Gamma$
consists of a \emph{single} point that is regular.
Therefore the one-dimensional cobordism $N_\Gamma$ cannot be compact
by the classification theorem of compact one-manifolds.

Applying the paramterized version of Proposition \ref{prop:MMAK}
under the given energy bound, we conclude that there exists a sequence
$\{(s_\alpha, u_\alpha)\}$ with $s_\alpha \to s_0$ and $0< s_0 <1$ and a \emph{non-empty} set
${\bf v} \cup {\bf w}$ of
bubbles ${\bf v}=\{v_i\}_{i=1}^{N_1}$, $ {\bf w} =\{w_j\}_{j=1}^{N_2}$ such that
$u_\alpha \in \MM_{K(s_\alpha)}(\overline{J},H;0)$
weakly converges to the cusp curve
\be\label{eq:uinfty}
u_\infty = u_0+ \sum_{i=1}^{N_1} v_i + \sum_{j=1}^{N_2} w_j.
\ee
Here $u_0 \in \CM_{K(s_0)}(\overline J,H)$, and $w_k$'s and $v_\ell$'s
are non-constant $J_0$-holomorphic discs and $J_{(K(s_0), {\tau_\ell}, t_\ell)}$-holomorphic
spheres for some $(\tau_\ell, t_\ell)$ respectively.
It follows from a dimension counting that $s_0$ cannot be either 0 or 1, because
the corresponding moduli spaces restricted thereto are Fredholm regular.
We also have the energy bound
\be\label{eq:EJKu}
E_{(J_{K(s_0)},H)} (u_0) + \sum_i \omega ([v_i]) + \sum_j \omega ([w_j]) \leq \| H \|.
\ee

%
Since the bubble set ${\bf v} \cup {\bf w}$ is not empty, the energy bound \eqref{eq:EJKu} implies
$$
A^U(L;M,\omega) \leq \|H\|
$$
by the definition \eqref{eq:AULMomega} of $A^U(L;M,\omega)$.

Combining Corollary \ref{cor:alternative} and the above analysis of
failure of compactness, we have proved
$$
\epsilon(U,L;M,\omega) \leq \|H\|
$$
for any Hamiltonian $H$ such that $\overline U \cap \phi_H^1(L) = \emptyset$.
By taking the infimum over all such $H$, we have obtained
$$
e(\overline U,L) \geq \epsilon(U,L;M,\omega) > 0.
$$
This finishes the proof of Theorem \ref{thm:enhanced}.
\end{proof}

\begin{rem}\label{rem:comparison} In this remark, we would like to compare the arguments used
above proof
with that of \cite{oh:mrl}. The main difference in the geometric circumstances between \cite{oh:mrl}
and the present case is as follows: In the former case, $L$ was displaceable, i.e.,
there was a Hamiltonian $H$ such that $\phi_H^1(L) \cap L =\emptyset$, while
in the present case $L$ is displaceable from an open subset $U$, i.e.,
$$
L \cap U \neq \emptyset, \quad \phi_H^1(L) \cap \overline U =\emptyset
$$
for open subset $U \subset M$ .

Even though a choice of an embedded path
$\gamma:[0,1] \to L$ and consideration of the one-dimensional
cobordism $N_\Gamma$ defined as above was also made
in the proof of the main theorem \cite{oh:mrl} (see p.902 therein), such a choice was
unnecessary for the purpose of \cite{oh:mrl} because it is enough to know non-compactness of the full
$(n+1)$-dimensional cobordism  to prove
the inequality $e(L,L) \geq A(L;M,\omega)$ in the scheme used therein. (As a matter of fact,
the author made such a consideration at that time having application
to the study of Maslov class obstruction in his mind. He did not
pursue this further realizing that such a consideration did not gain much in that
it does not give rise to anything significant for the study of Maslov class obstruction
beyond that of Polterovich \cite{polterov:maslov}.) However consideration here of
this one-dimensional cobordism through a point $p \in L \cap U$,
accompanied by the construction of the invariant $A(U,L;M,\omega)$ relative to the open subset $U$,
is a crucial ingredient needed for the proof of Main Theorem.
\end{rem}

\end{document}